\documentclass[11pt]{amsart}
\usepackage[margin=1.0in]{geometry}
\renewcommand{\qedsymbol}{$\blacksquare$}
\usepackage{amsmath}%
\usepackage{amsfonts}%
\usepackage{amssymb}%
\usepackage{graphicx}
\usepackage{bbm}

\usepackage{xcolor}

\newtheorem{theorem}{Theorem}
\theoremstyle{plain}

\newtheorem{observation}[theorem]{Observation}

\newtheorem{lemma}[theorem]{Lemma}

\newtheorem{proposition}[theorem]{Proposition}

\numberwithin{equation}{section}
\newtheorem*{acknowledgement*}{Acknowledgement}
\newtheorem*{algorithm*}{Algorithm}
\newtheorem*{axiom*}{Axiom}
\newtheorem*{case*}{Case}
\newtheorem*{claim*}{Claim}
\newtheorem*{observation*}{Observation}
\newtheorem*{conclusion*}{Conclusion}
\newtheorem*{condition*}{Condition}
\newtheorem*{conjecture*}{Conjecture}
\newtheorem*{corollary*}{Corollary}
\newtheorem*{criterion*}{Criterion}
\newtheorem*{definition*}{Definition}
\newtheorem*{example*}{Example}
\newtheorem*{exercise*}{Exercise}
\newtheorem*{lemma*}{Lemma}
\newtheorem*{notation*}{Notation}
\newtheorem*{problem*}{Problem}
\newtheorem*{proposition*}{Proposition}
\newtheorem*{remark*}{Remark}
\newtheorem*{solution*}{Solution}
\newtheorem*{summary*}{Summary}

\newcommand{\eps}{\varepsilon}
\newcommand{\wh}[1]{\widehat{#1}}
\newcommand{\NN}{\mathbb{N}}
\newcommand{\RR}{\mathbb{R}}
\newcommand{\uno}{\mathbbm{1}}

\newcommand{\alocc}[2]{|\!|#1|\!|_{#2}}
\newcommand{\occ}[2]{|#1|_{#2}}

\begin{document}

\title{From randomness in two symbols \linebreak to randomness in three symbols}
%From randomness in base two to randomness in base three}
\author[1]{Ariel Zylber \\ \today}
\address{Departamento de Computaci\'on,
Facultad de Ciencias Exactas y Naturales,
Universidad de Buenos Aires }
\email{ azylber@dc.uba.ar}
%\date{\today}

\begin{abstract}
In 1909 Borel defined normality as a notion of randomness 
of the digits of the representation of a real number over certain base 
(fractional expansion). If we think the representation of a number over a base as 
% a 
an
infinite sequence of symbols from a finite alphabet $A$, 
we can define normality directly for words of symbols of $A$: 
A word $x$ is normal to the alphabet $A$ if every finite block of 
symbols from $A$ appears with the same asymptotic frequency in $x$ 
as every other block of the same length. Many examples of normal words 
have been found since its definition, being Champernowne in 1933 the first 
to show an explicit and simple 
%example. 
instance.
Moreover, it has been characterized how we can select subsequences 
of a normal word $x$ preserving its normality, always leaving %fixed 
the alphabet $A$
fixed.
In this 
% word 
work
we consider the dual problem which consists of inserting symbols 
%on 
in
infinite positions of a given 
%words,
word,
in such a way that normality is preserved. 
Specifically, given a symbol $s$ that is not present on the original alphabet $A$ 
and given 
%$x$, a normal word to the alphabet $A$ 
a word $x$ that is normal to the alphabet $A$ 
we solve how to insert the symbol $s$ 
%on 
in
infinite positions of the word $x$ such that the resulting word 
is normal to the expanded alphabet $A\cup \{s\}$.
\end{abstract}
\maketitle

%\tableofcontents
\section{Introduction and statement of results}

In 1909, Borel \cite{borel1909probabilites} defined normality as a notion of randomness
% in
of
 the digits 
of the fractional expansion of a real number over some base. 
Since then many examples of normal words have been found,
 Champernowne\cite{Champernowne:1933} in 1933 was the first  
to show an explicit and simple 
%example. 
instance,
$$
12345678910111213141516171819202122232425262728\ldots
$$
the concatenation of all the natural numbers in the natural order is a  normal word for  the alphabet 
$A=\{0,1,\ldots,9\}$.  
Moreover, it has been characterized how we can select subsequences 
of a normal word $x$ preserving its normality, always leaving fixed the alphabet 
$A$, see \cite{agafonov1968normal, kamae1975normal, vandehey2016uncanny}.

In this work we 
%solve 
consider
how normality of words is affected when we add new symbols to the alphabet. 
%
%Clearly, we lose the normality of a word if we add a new symbol to the alphabet 
%since the word contains no appearances of this new symbol. 
Clearly, if a word $x$ is normal to a given alphabet $A$ it is not normal to an alphabet $A'$
that  results from adding a new symbol to $A$,
because  the word $x$ contains no appearances of this new symbol. 
A natural question that comes up is if it is possible to insert 
%some 
occurrences of this new symbol 
%in the middle of 
along
the word $x$ to make it normal in the 
%new 
expanded alphabet. 
We give a positive answer of this question in Theorem~\ref{maintheo}. 
%In this theorem 

Fix an alphabet $A$ and a new symbol $s$.
For any given normal word $x$ in $A^\omega$ the proof of Theorem~\ref{maintheo}
%we give
gives a way of inserting occurrences of the new symbol  $s$ 
%in the word 
along the word $x$
that depends on the speed of convergence of  normality of the word $x$.
%The main idea behind the theorem is that we take finite substrings of Champer\-nowne like words over the extended alphabet that have the property of having a good distribution of the new symbol on it (null discrepancy over certain length). We will always talk about discrepancy of aligned occurrences of a word. We take advantage of that distribution and try to replicate it in the normal word so the occurrences of the new symbol in the extended word match the pattern of the Champernowne like word.\\
The proof is purely combinatorial and it is completely elementary except for the 
use of  the characterization of normality given by Piatetski-Shapiro 
\cite{Piatetski-Shapiro:1951, bugeaud2012distribution}
also known as the Hot Spot Lemma.

The main idea in the proof of Theorem~\ref{maintheo}
is to use  a Champernowne-like word in the expanded alphabet as a reference for insertion 
of the new symbol $s$ in the given normal word $x$.
We call the discrepancy of a finite word $w$  with respect to the length $\ell$
to the maximum difference between the expected frequency and the actual frequency  in $w$ 
of any block of $\ell$ digits.
The key ingredient of the proof of Theorem~\ref{maintheo}
 is given by  Lemma~\ref{fund} where we prove 
that if the discrepancy of a finite word $w$ in the original alphabet  
with respect to a given  length 
is low enough then  inserting occurrences of the new symbol in $w$
%in  
according to the pattern of a 
%finite 
Champernowne-like word
yields an expanded word
 with also low discrepancy
but now with respect to an exponentially shorter length.
The proof of this lemma relies on 
bounding the number of 
%aligned 
occurrences
of a word in the expanded word.
%In the fourth section we finally give a proof of the theorem. 
In the proof of Theorem~\ref{maintheo}  
we  take consecutive segments of the original word $x$, of increasing length,
and expand each of them according to the pattern of digits  given by 
a Champernowne-like
% finite
 word.
The difficulty  here is in determining the appropriate lengths of these segments.
They have to be long enough so that their discrepancy catches 
up with the discrepancy of the Champernowne-like word.
At last, an application of  Piatetski-Shapiro's characterization 
of normality allows us to conclude the normality of the expanded word.

\subsection{Primary definitions}
%To recall this definition over the context of infinite words.
% over an alphabet we first introduce some notation.\\
We call an alphabet to a finite set $A$ of symbols. 
Given an alphabet $A$, we write $A^k$ for the set of all words of length $k$, 
$A^*$ for the set of all finite words and $A^{\omega}$ for the set of all infinite words of $A$. 
Therefore, $(A^k)^*$ denotes the set of all finite words composed of the words 
of length $k$ of $A$ as symbols, or equivalently, the set of all finite words of length multiple of $k$.

The length of a finite word $v$ is denoted by $|v|$. 
Given two words $u$ and $v$ with $u$ finite, we denote $uv$ 
to the word resulting of concatenating $u$ and $v$. 
The position of symbols in words are numbered starting from $1$. 
For a word $v$, we denote $v[i,j]$ as the substring of $v$ from position $i$ to position $j$. 
We call $v[i]$ to the symbol corresponding to the $i$-th position of $v$. 
We call substring of a word $v$ to a word of the form $v[i,j]$ for some $i, j \in \NN$ 
such that $1 \leq i \leq j \leq |v|$ and subsequence of $v$ to a 
word of the form $v[i_1]v[i_2] \ldots v[i_k]$ for some $i_1, i_2 ,\ldots, i_k \in \NN$
 with $i_1 < i_2 < \ldots < i_k \leq |v|$.

Given some alphabet $A$ and $u, v \in A^*$, we %denote 
write 
$$
\alocc{u}{v} %||u||_v 
= |\{i \leq |u|-|v|+1: u[i,i+|v|-1] = v \text{ and } i \equiv 1 \mod{|v|}\}| 
$$
for the number of aligned occurrences of $v$ in $u$. 
Thus, if we split the word $u$ in 
consecutive strings
of length $|v|$  and possibly a shorter last string, 
$\alocc{u}{v}$ is the number of those strings that coincide with $v$.

With this notation we can state the  formulation 
of  normality that is most convenient for 
to solve our problem.
A thorough  presentation of normality can be read from 
the monographs \cite{bugeaud2012distribution,BecherCarton2017}.

\begin{definition*}[Normality to a given alphabet]
Given an alphabet $A$ and some word $u \in A^{\omega}$, 
we say that $u$ is simply normal to length $\ell$ if every $v \in A^{\ell}$ verifies that 
$$
\lim_{n \rightarrow \infty} \frac{\alocc{u[1,\ell n]}{v}}{n} = \frac{1}{|A|^{\ell}}.
$$
We say that $u$ is normal if it is simply normal to every length $\ell \in \NN$.
\end{definition*}

From now on, we fix a base $b$ and we define $A = \{0, 1, \ldots, b-1 \}$ 
and $\widehat{A} = \{0, 1, \ldots, b \}$,
%as 
the alphabets whose symbols are the digits in base $b$ and base $b+1$ respectively. 
We write  $v \upharpoonright n$  to denote  $v[1,n]$ which is  
the word  consisting of the first $n$ symbols of $v$, 
and  we write $v \upharpoonleft n$ to denote 
the word   that results from removing the last  $n$ symbols of $v$.

\begin{definition*}[reduction operator]
We define the reduction operator $r: \widehat{A}^* \rightarrow A^*$ 
as the operator that removes the  symbols
 $b$
from a  word in $ \widehat{A}^*$.
Precisely, given a word $v\in \widehat{A}^*$,
$$
v = v_1v_2\ldots v_k
$$
 where $v_i$ is the $i$-th symbol of $v$, then 
$$
r(v) = v_{r_1}v_{r_2}v_{r_t}
$$
 where 
$$
t = |v| - \alocc{v}{b}$$ and 
$$
r_i = \min(\{j \in \mathbb{N}: |v \upharpoonright j| - \alocc{v \upharpoonright j}{b} = i\}).
$$
We define the reduction operator $r$ on infinite words $v \in \widehat{A}^{\omega}$ in a similar way.
\end{definition*}
%Clearly, $r$ is a retraction of $e_n$ for all $n \in \mathbb{N}$, that is, 
%\[
%r \circ e_n = id.
%\] 

\subsection{The main theorem}

\begin{theorem}\label{maintheo}
Let $v \in A^{\omega}$ be a normal word then there exists some normal word 
$\wh{v} \in \wh{A}^{\omega}$ such that $r(\wh{v}) = v$.
\end{theorem}

Before giving the proof of  Theorem~\ref{maintheo} we need
some intermediate results.

\section{Tools and lemmas}

We define here in a precise way how we expand a word according to 
the pattern of a Champernowne-like word.

\begin{definition*}[Champernowne-like words]
For each $n \in \mathbb{N}$, let $w_n$ be the word 
consisting of the concatenation in lexicographical order of all the words of $\widehat{A}^n$.
\end{definition*}
Thus, for $A=\{0,1\} $,
$w_3 = 
000
001
010
011
100
101
110
111$.
\begin{definition*}[The wildcard operator]
Let $B = \{b, \star\}$ the alphabet consisting of only the symbols $'b'$ and $'\star'$.
We define the \textit{wildcard operator} $(\star): \widehat{A}^* \rightarrow B^*$ 
as the operator that given $v \in \widehat{A}^*$ replaces all its symbols different from~$'b'$ 
with a wildcard~$'\star'$.
Formally, if 
$$
v = v_1 v_2 \ldots v_k
$$ 
where $v_i$ is the $i$-th symbol from $v$, 
then, 
$$
(\star)(v) = v^{\star}_1v^{\star}_2 \ldots v^{\star}_k
$$ 
where 
$$ 
v^{\star}_i = \begin{cases} b, & \mbox{if } v_i = b \\ {\star}, & \mbox{otherwise} \end{cases}
$$
We write $v^{\star} = ({\star})(v)$.
\end{definition*}
 It follows easily that if 
$u, v \in \widehat{A}^*$ then $(uv)^{\star} = u^{\star}v^{\star}$.

%\V
%We write  $v \upharpoonright n$  to denote  $v[1,n]$ 
%%(this is the word  consisting of the first $n$ symbols of $v$ 
%and  we write $v \upharpoonleft n$ to denote 
%the word that consists of the word  that results from removing the first  $n$ symbols of $v$.
%\A

\begin{definition*}[The expansion of order $n$ of a given word]
For each $n \in \mathbb{N}$ we let
% let $\ell_n = |w_n|_{\star}$ 
%be the number of wildcard digits in $w^{\star}_n$ and $\widehat{\ell}_n = |w_n|$.
\begin{align*}
\ell_n& = \alocc{w_n}{\star}
\\
\widehat{\ell}_n& = |w_n|.
\end{align*}
%For each $i \in \mathbb{N}, 1 \leq i \leq \widehat{l}_n$ let $$wi_n(i) = |\{ j \leq i : (w_n)_j = {\star} \}|$$ Thus, $wi_n(i)$ counts the number of wilcards in $w_n$ up to the $i$-th character.\\
For each $i \in \mathbb{N}$ such that
$ 1 \leq i \leq \widehat{\ell}_n$ define
$$
m(n,i) = |\{ j \leq i : (w_n)_j = {\star} \}| 
= \alocc{w_n \upharpoonright i}{\star}.
$$ 
Thus, $m(n,i)$ counts the number of wildcards in $w_n$ up to the $i$-th symbol.

The expansion 
$e_n: A^{\ell_n} \to \widehat{A}^{\widehat{\ell}_n}$ is
%Given a word $v \in A^{\ell_n}$ we define   the extension 
%of order $n$ of $v$ 
%as $e_n(v) = \widehat{v} \in \widehat{A}^{\widehat{\ell}_n}$ 
such that, if 
$$
v = v_1v_2 \ldots v_{\ell_n}
$$ 
then 
$$
\widehat{v} = \widehat{v}_1\widehat{v}_2 \ldots \widehat{v}_{\widehat{\ell}_n}
$$
%where $$\widehat{v}_i = \begin{cases} b, & \mbox{if } (w_n)_i = b \\ v_{wi_n(i)}, & \mbox{otherwise} \end{cases}$$
where 
$$
\widehat{v}_i = \begin{cases} b, & \mbox{if } (w_n)_i = b \\ v_{m(n,i)}, & \mbox{otherwise}.
\end{cases}
$$
%Hence,
%$e_n(v)$ takes $w_n$ and replaces each wildcard symbol with each symbol of $v$, in order.
\end{definition*}

Thus, given a word $v \in A^{\ell_n}$,  
the  expanded word $e_n(v)$  is obtained as follows: 
take  $w_n$, replace all its symbols different from $b$ by a wildcard symbol,
and then replace in each  wildcard symbol with the symbols of $v$ in order.
Clearly,  $v$ is a subsequence of $e_n(v)$ and the only digits 
that are not part of that subsequence are all $b$'s.

We can extend $e_n$ to $(A^{\ell_n})^*$ by concatenating the expansion
of each block of $\ell_n$ digits. 
Namely, if $v \in (A^{\ell_n})^*$ such that 
$$
v = v_1v_2 \ldots v_k
$$ 
where $|v_i| = \ell_n$ for all $1 \leq i \leq k$, then 
$$
e_n(v) = e_n(v_1)e_n(v_2) \ldots e_n(v_k).
$$

%\begin{definition*}[reduction operator]
%We define the reduction operator $r: \widehat{A}^* \rightarrow A^*$ 
%as the operator that removes the  symbols
% \V $b$
%from a  word in $ \widehat{A}^*$.\A 
%Precisely, given a word $v\in \widehat{A}^*$,
%$$
%v = v_1v_2\ldots v_k
%$$
% where $v_i$ is the $i$-th symbol of $v$, then 
%$$
%r(v) = v_{r_1}v_{r_2}v_{r_t}
%$$
% where 
%$$
%t = |v| - |v|_b$$ and 
%$$
%r_i = \min(\{j \in \mathbb{N}: |v \upharpoonright j| - |v \upharpoonright j|_b = i\}).
%$$
%We define the reduction $r$  infinite words $v \in \widehat{A}^{\omega}$ in a similar way.
%\end{definition*}
Clearly,  the reduction $r$ is a retraction of $e_n$ for all $n \in \mathbb{N}$, that is, 
\[
r \circ e_n = id.
\]

The next observations follow from the definitions.
%\subsection{The discrepancy of an extension}
%We establish here the fundamental lemma for the proof.

\begin{observation}\label{valorell} 
$\wh{\ell}_n = n(b+1)^n$ and $\ell_n = nb(b+1)^{n-1}$ for all $n \in \NN$. 
\end{observation}
\begin{proof}
Since there are $(b+1)^n$ different words of length $n$ 
using $b+1$ symbols and each word has length $n$ we get $\wh{\ell}_n = n(b+1)^n$.
Since each symbol appears the same number of times in $w_n$ then 
$\alocc{w_n}{b} = n(b+1)^{n-1}$. It follows that 
$$
\ell_n = |w_n| - \alocc{w_n}{b} = n(b+1)^n - n(b+1)^{n-1} = nb(b+1)^{n-1}.
$$ 
%as desired
\end{proof}

 Given some alphabet $A$, we denote $\uno$ to the indicator function 
of the diagonal elements of $A^* \times A^*$. Namely, we define 
$\uno: A^* \times A^* \rightarrow \NN$ as 
$$
\uno(x,y) = \begin{cases} 1, & \mbox{if } x = y \\ 0, & \mbox{otherwise} \end{cases}
$$
We denote $\uno(x,y)$ as $\uno(x = y)$.

\begin{observation}\label{cuentasubs} Given an alphabet $C$ with $|C| = k$, 
some $v \in C^n$, some $m \in \NN$ such that $m > n$ and some $i \in \NN$ 
such that $0 \leq i \leq m-n$, then 
$$
\sum_{u \in A^m} \uno(u[i+1, i+n] = v) = k^{m-n}.
$$
\end{observation}
\begin{observation}\label{obspartes} Given an alphabet $C$ with $|C| = k$, some $v \in C^n$ and $u \in (C^n)^*$ then 
$$
\alocc{u}{v} = \sum_{i = 0}^{|u|/n-1} \uno(u[in+1, in+n] = v).
$$
\end{observation}

\begin{observation}\label{igualdadext} 
Given $v, w \in \wh{A}^*$ then $v = w$ if and only if $v^{\star} = w^{\star}$ and $r(v) = r(w)$.
\end{observation}

\begin{observation}\label{appchamp} If $v \in B^n$ then $\alocc{w^{\star}_n}{v} = b^{\alocc{v}{\star}}$.
\end{observation} 

\begin{observation}\label{appext} If $v \in A^{\ell_n}$ and $w \in (A^{\ell_n})^*$ then $\alocc{w}{v} = \alocc{e_n(w)}{e_n(v)}$.
\end{observation} 

\begin{lemma}\label{sumafund}
Given $w \in \wh{A}^n$ then 
$$
\sum_{u \in A^{\ell_n}} \alocc{e_n(u)}{w} = b^{\ell_n}.
$$ 
\end{lemma}
\begin{proof}
By Observation~\ref{obspartes} we have  
$$
\alocc{e_n(u)}{w} = \sum_{i = 0}^{|u|/n-1} \uno(e_n(u)[in+1, in+n] = w)
$$
for all $u \in A^{\ell_n}$.
Applying Observation~\ref{igualdadext} we get 
$\alocc{e_n(u)}{w} $ is equal to
\begin{equation}\label{sumabetas}
\sum_{i = 0}^{|u|/n-1} \uno\Big((e_n(u)[in+1, in+n])^{\star} 
= w^{\star}\Big)\uno\Big(r(e_n(u)[in+1, in+n]) = r(w)\Big).
\end{equation}
Analyzing the definition of $({\star})$ we get that 
$$
(e_n(u)[in+1, in+n])^{\star} = 
 e_n(u)^{\star}[in+1, in+n] = 
 (w^{\star}_n)[in+1, in+n].
$$ 
By Observation~\ref{appchamp} we conclude that 
$$
\alocc{w^{\star}_n}{w^{\star}} = b^{\alocc{w^{\star}}{\star}} = b^{|w|-\alocc{w}{b}}.
$$ 
This means that there are exactly $b^{|w|-\alocc{w}{b}}$ terms of the sum in which 
$$
\uno\Big((e_n(u)[in+1, in+n])^{\star} = w^{\star}\Big) = 1.
$$ 
Let 
$$
I = \{0 \leq i < \ell_n/n: \uno((w^{\star}_n)[in+1, in+n] = w^{\star}) = 1 \}.
$$
 be the set of indexes where the first term of the product does not vanish. 
Notice that $I$ does not depend on $u$.

Analyzing the second term of the product, we observe that 
$$
\uno\Big(r(e_n(u)[in+1, in+n]\Big) = r(w)) = \uno\Big(u[m(n,in)+1, m(n,in+n)] = r(w)\Big).
$$ 
applying this we reduce (\ref{sumabetas}) to 
\begin{equation}\label{sumabetados}
\alocc{e_n(u)}{w} 
= \sum_{i \in I} \uno\Big(u[m(n,in)+1, m(n,in+n)] 
= r(w)\Big).
\end{equation}
%Notice that 
Since $i \in I$ we have that 
$(e_n(u)[in+1, in+n])^{\star} = w^{\star}$, which implies that 
$$
\left|u[m(n,in)+1, m(n,in+n)]| =|r(e_n(u)[in+1, in+n])\right| = |r(w)|.
$$
Summing (\ref{sumabetados}) over all $u \in A^{\ell_n}$ we get
$$
\sum_{u \in A^{\ell_n}} \alocc{e_n(u)}{w} 
= \sum_{u \in A^{\ell_n}} \sum_{i \in I} \uno(u[m(n,in)+1, m(n,in+n)] 
= r(w)).
$$
And applying Observation~\ref{cuentasubs} we get
$$
\sum_{u \in A^{\ell_n}} \alocc{e_n(u)}{w} 
= \sum_{i \in I} b^{\ell_n-|r(w)|} 
= b^{|w|-\alocc{w}{b}} b^{\ell_n-|r(w)|}.
$$
And noticing that by definition of $r$ we have that 
$|r(w)| = |w|-\alocc{w}{b}$ this gives us the desired result 
$$
\sum_{u \in A^{\ell_n}} \alocc{e_n(u)}{w} = b^{\ell_n}.
$$
\end{proof}

\subsection{On discrepancies}
Here we introduce a definition of discrepancy for finite words and 
we relate the discrepancy of a word and the discrepancy of  the expanded word.
We also consider the concatenation of  a sequence of words and 
we bound the discrepancy of the resulting word  in terms of the discrepancies of the individual words.
Most of the bounds that we give can be improved 
but these simple versions will be enough for the proof of Theorem~\ref{maintheo}. 

Given some alphabet $A$, some word $u \in A^*$ 
and a fixed length $\ell \in \NN$, for a word $v \in A^{\ell}$ 
 the frequency of aligned occurrences of $v$ in $u$ 
over all aligned substrings of length $\ell$ in $u$ is
$$
\frac{\alocc{u}{v}}{\lfloor|u|/\ell \rfloor}.
$$
We can measure how far is this frequency from the case where 
all words of length~$\ell$ are equiprobable by
$$
\left|\frac{\alocc{u}{v}}{\lfloor|u|/\ell \rfloor} - \frac{1}{|A|^{\ell}}\right|.
$$

The discrepancy  of a word $u$ in $A^* $ for a length $\ell$
is the maximum 
of this distance among all $v \in A^{\ell}$ and we denote it by $\Delta_{A,\ell}(u)$.

\begin{definition*} [Discrepancy of a finite word for a given length $\ell$]
$$
\Delta_{A,\ell}(u) = 
\max_{v \in A^{\ell}}\left(\left|\frac{\alocc{u}{v}}{\lfloor|u|/\ell \rfloor} - \frac{1}{|A|^{\ell}}\right|\right).
$$
\end{definition*} 

An easy equivalence is that $u$ is simply normal to length $\ell$ if and only if
$$
\lim_{n \rightarrow \infty} \Delta_{A,\ell}(u[1,n]) = 0.
$$
and therefore $u$ is normal if and only if this limit is valid for every length $\ell \in \NN$.

Let $u\in A^*$, let $\ell$ be a length and let $\eps$ be a real umber between $0$ and $1$.
Then  it follows that 
$$
\Delta_{A,\ell}(u) < \eps
$$
is equivalent to have for all $v \in A^{\ell}$,
$$
\lfloor|u|/\ell \rfloor\left(\frac{1}{|A|^{\ell}} - \eps\right) < 
\alocc{u}{v} < \lfloor|u|/\ell \rfloor\left(\frac{1}{|A|^{\ell}} + \eps\right).
$$

\begin{lemma}[Main Lemma]\label{fund}
For each $n \in \mathbb{N}$ there exists a constant 
$c_n \in \RR$ with $c_n > 0$ such that for every $\varepsilon > 0$ 
and every word $v \in (A^{\ell_n})^*$ if 
\begin{equation}\label{bounddisc}
\Delta_{A,\ell_n}(v) < \varepsilon
\end{equation}
 then 
$$
\Delta_{\wh{A},n}(e_n(v)) < c_n\varepsilon.
$$
\end{lemma}
\begin{proof}
Let $w \in \wh{A}^n$ be any word of length $n$,
 then 
$$
\alocc{e_n(v)}{w} = \sum_{\wh{u} \in \wh{A}^{\wh{\ell}_n}} \alocc{e_n(v)}{\wh{u}} \alocc{\wh{u}}{w}.
$$
By the definition of $e_n$, the blocks of length $\wh{\ell}_n$ of $e_n(v)$ 
are of the form $e_n(v_i)$ for some $v_i \in A^{\ell_n}$. Then, the only non-zero terms of the sum can be the ones where $\wh{u}$ is in the image of $e_n$, and since $e_n$ is injective we can change the sum to iterate over the $e_n(u)$ for $u \in A^{\ell_n}$. It follows that 
$$
\alocc{e_n(v)}{w} = \sum_{u \in A^{\ell_n}} \alocc{e_n(v)}{e_n(u)} \alocc{e_n(u)}{w}.
$$
By Observation~\ref{appext} it reduces to
$$\alocc{e_n(v)}{w} = \sum_{u \in A^{\ell_n}} \alocc{v}{u} \alocc{e_n(u)}{w}.
$$
Applying (\ref{bounddisc}) we get
$$\alocc{e_n(v)}{w} < \sum_{u \in A^{\ell_n}} \frac{|v|}{|u|}\left(\frac{1}{b^{|u|}} + \eps \right) \alocc{e_n(u)}{w} = \frac{|v|}{\ell_n}\left(\frac{1}{b^{\ell_n}} + \eps \right) \left(\sum_{u \in A^{\ell_n}} \alocc{e_n(u)}{w}\right).
$$
Using Observation~\ref{sumafund} we get
$$\alocc{e_n(v)}{w} < \frac{|v|}{\ell_n}\left(\frac{1}{b^{\ell_n}} + \eps \right) b^{\ell_n} = \frac{|v|}{\ell_n}\left(1 + b^{\ell_n}\eps \right).
$$
Multiplying by $\frac{|w|}{|e_n(v)|} = \frac{n}{|e_n(v)|}$ on both sides % give us 
we obtain
\begin{equation}
\label{venv}\frac{|w|}{|e_n(v)|}\alocc{e_n(v)}{w} < \frac{n|v|}{\ell_n |e_n(v)|}\left(1 + b^{\ell_n}\eps \right).
\end{equation}
Since $v \in (A^{\ell_n})^*$ we can write $v$ as 
$$
v = v_1v_2\ldots v_t
$$ 
where each $v_i$ satisfies $|v_i| = \ell_n$. 
Then $|v| = t \ell_n$ and 
$$
e_n(v) = e_n(v_1)e_n(v_2)\ldots e_n(v_t)
$$
 where $|e_n(v_i)| = \wh{\ell}_n$ for all 
$1 \leq i \leq t$. So, we conclude that $|e_n(v)| = t\wh{\ell}_n$.

Using this on (\ref{venv}) we get
$$
\frac{|w|}{|e_n(v)|}\alocc{e_n(v)}{w} 
< 
\frac{n t \ell_n}{\ell_n t \wh{\ell}_n}\left(1 + b^{\ell_n}\eps \right)
$$
using Observation~\ref{valorell} we can replace the value of $\wh{\ell}_n$ and get
$$
\frac{|w|}{|e_n(v)|}\alocc{e_n(v)}{w} < \frac{n}{n (b+1)^n}\left(1 + b^{\ell_n}\eps \right) 
= \frac{1}{(b+1)^n} + \frac{b^{\ell_n}}{(b+1)^n}\eps.
$$
By a similar argument we get the inequality
$$
\frac{|w|}{|e_n(v)|}\alocc{e_n(v)}{w} > \frac{1}{(b+1)^n} - \frac{b^{\ell_n}}{(b+1)^n}\eps.
$$
These two inequalities imply that 
$$
\Delta_{\wh{A},n}(e_n(v)) < \frac{b^{\ell_n}}{(b+1)^n}\eps.
$$
The desired result follows taking 
\[
c_n = \frac{b^{\ell_n}}{(b+1)^n}.
\]
\end{proof}

\subsection{Some other useful results}

\begin{proposition}\label{mitad}
Given a finite alphabet $C$ with $|C| = k$ and some $m, n \in \NN$. 
We have that for each word $v \in (C^{mn})^*$ and $\eps \in \RR$ with $\eps > 0$ such that 
\begin{equation}\label{discmitad}\Delta_{C, mn}(v) < \eps\end{equation} 
then 
$$
\Delta_{C, n}(v) < k^{(m-1)n}\eps.
$$
\end{proposition}
\begin{proof}
Let $w \in C^n$ be any word of length $n$. We have that 
$$
\alocc{v}{w} = \sum_{u \in C^{mn}} \alocc{v}{u} \alocc{u}{w}.
$$ 
Using \ref{discmitad} we get 
$$
\alocc{v}{w} < \sum_{u \in C^{mn}} \frac{|v|}{|u|}\left(\frac{1}{k^{mn}} + \eps\right) \alocc{u}{w}.
$$
Using Observation~\ref{obspartes} we get 
$$
\alocc{v}{w} <  \frac{|v|}{mn}\left(\frac{1}{k^{mn}} + \eps\right) \sum_{u \in C^{mn}} \sum_{i = 0}^m \uno(u[in+1, in+n] = v).
$$ 
Using Observation~\ref{cuentasubs} we get 
$$
\alocc{v}{w} < \frac{|v|}{mn}\left(\frac{1}{k^{mn}} + \eps\right) \sum_{i = 0}^m k^{mn-n} = \frac{|v|}{n}\left(\frac{1}{k^{n}} + k^{(m-1)n}\eps\right).
$$
\end{proof}

\begin{proposition}\label{sufijo}
Given a finite alphabet $C$, some $n \in \NN$ and $u,v \in (C^n)^*$, if 
\begin{equation}\label{sufdiscu}\Delta_{C, n}(u) < \eps\end{equation}
 and 
\begin{equation}\label{sufdiscuv}\Delta_{C, n}(uv) < \eps\end{equation}
 then 
$$
\Delta_{C, n}(v) < \frac{|uv|+|u|}{|v|}\eps.
$$
\end{proposition}
\begin{proof}
Let $w \in C^n$ be any word of length $n$. Then, 
$$
\alocc{v}{w} = \alocc{uv}{w} - \alocc{u}{w}.
$$ 
Using (\ref{sufdiscu}) and (\ref{sufdiscuv}) we get 
$$
\alocc{v}{w} < \frac{|uv|}{|w|}\left(\frac{1}{k^{|w|}} +
 \eps\right) - \frac{|u|}{|w|}\left(\frac{1}{k^{|w|}} - \eps\right)
$$
 which using $|uv| = |u| + |v|$ is equivalent to  
$$
\alocc{v}{w} <\frac{|v|}{|w|}\frac{1}{k^{|w|}} + \frac{|uv|+|u|}{|w|}\eps
$$
 which is equivalent to 
$$
\alocc{v}{w} < \frac{|v|}{|w|}\left(\frac{1}{k^{|w|}} + \frac{|uv|+|u|}{|v|}\eps\right).
$$ 
In a similar way we can conclude 
$$
\alocc{v}{w} > \frac{|v|}{|w|}\left(\frac{1}{k^{|w|}} - \frac{|uv|+|u|}{|v|}\eps\right).
$$
Since both inequalities are valid for all $w \in C^n$ we conclude the result.
\end{proof}

\begin{proposition}\label{concat}
Given a finite alphabet $C$, some $n \in \NN$ and $u,v \in (C^n)^*$, if 
\begin{equation}\label{concatdiscu}\Delta_{C, n}(u) < \eps\end{equation}
 and 
\begin{equation}\label{concatdiscv}\Delta_{C, n}(v) < \frac{|uv|+|u|}{|v|}\eps\end{equation}
 then 
$$
\Delta_{C, n}(uv) < 3\eps.
$$
\end{proposition}
\begin{proof}
Let $w \in C^n$ be any word of length $n$. Then, 
$$
\alocc{uv}{w} = \alocc{u}{w} + \alocc{v}{w}.
$$ 
Using (\ref{concatdiscu}) and (\ref{concatdiscv}) we get 
$$
\alocc{uv}{w} < \frac{|u|}{|w|}\left(\frac{1}{k^{|w|}} + \eps\right) + \frac{|v|}{|w|}\left(\frac{1}{k^{|w|}} + \frac{|uv|+|u|}{|v|}\eps\right)
$$
 which using $|uv| = |u| + |v|$ is equivalent to 
$$
\alocc{uv}{w} <\frac{|u|+|v|}{|w|}\frac{1}{k^{|w|}} + \frac{3|u|+|v|}{|w|}\eps,
$$
 and since $3|u|+|v| < 3(|u|+|v|)$ we get 
$$
\alocc{uv}{w} < \frac{|u|+|v|}{|w|}\left(\frac{1}{k^{|w|}} + 3\eps\right).
$$ 
In a similar way we can conclude 
$$
\alocc{uv}{w} > \frac{|u|+|v|}{|w|}\left(\frac{1}{k^{|w|}} - 3\eps\right).
$$
Since both inequalities are valid for all $w \in C^n$ we conclude the result.
\end{proof}

%\begin{proposition}\label{concatfuerte}
%Given a finite alphabet $C$, some $n \in \NN$ and $u,v \in (C^n)^*$, if 
%\begin{equation}\label{concatfdiscu}\Delta_{C, n}(u) < \eps\end{equation}
 %and 
%\begin{equation}\label{concatfdiscv}\Delta_{C, n}(v) < \eps\end{equation}
 %then 
%$$\Delta_{C, n}(uv) < \eps$$
%\end{proposition}
%\begin{proof}
%Let $w \in C^n$ be any word of length $n$. Then, 
%$$||uv||_w = ||u||_w + ||v||_w$$ 
%Using \ref{concatfdiscu} and \ref{concatfdiscv} we get 
%$$||uv||_w < \frac{|u|}{|w|}\left(\frac{1}{k^{|w|}} + \eps\right) + \frac{|v|}{|w|}\left(\frac{1}{k^{|w|}} + \eps\right)$$
 %which using $|uv| = |u| + |v|$ is equivalent to 
%$$||uv||_w < \frac{|uv|}{|w|}\left(\frac{1}{k^{|w|}} + \eps\right)$$
%In a similar way we can conclude 
%$$||uv||_w > \frac{|uv|}{|w|}\left(\frac{1}{k^{|w|}} - \eps\right)$$
%Since both inequalities are valid for all $w \in C^n$ we conclude the result.
%\end{proof}

Our analysis so far focuses in aligned occurrences of a given word  in an expanded word.
For a technical reason 
the proof of Theorem~\ref{maintheo} 
 needs to consider  the number of non-aligned occurrences of  any given word in the 
constructed expanded word.
We define the number of non-aligned occurrences of a word $v$ in a word $u$ as
$$
\occ{u}{v} = |\{i \leq |u|-|v|+1: u[i,i+|v|-1] = v\}| 
$$

Notice that for every symbol $b\in A$ and for every word $u\in A^* $,
$$
\occ{u}{b} =\alocc{u}{b}.
$$

The following  proposition gives the needed result.

\begin{proposition}\label{alinanoalin}
Given a finite alphabet $C$, some $n, m \in \NN$ with $m < n$ some $u \in (C^n)^*$ and $v \in C^m$, if 
\begin{equation}\label{alindisc}\Delta_{C, n}(u) < \eps\end{equation}
 then 
$$
\occ{u}{v} < |u| \left(\frac{m-1}{n} + \frac{1}{|C|^m} + |C|^{n}\eps\right) - (m-1).
$$
\end{proposition}

\begin{proof}
For every pair of consecutive blocks of length $n$ in $u$ there 
are exactly $m-1$ substrings of length $m$ that are not fully
 contained in one of these blocks. Since there are $|u|/n$ blocks of length $n$ in $u$, 
there are $(|u|/n-1)(m-1)$ substrings of length $m$ not fully contained in one of the blocks. 
This gives us the following bound on the number of appearances of $v$ in $u$:
\begin{align*}
\occ{u}{v}
&\leq
 (|u|/n-1)(m-1) + \sum_{i=0}^{|u|/n-1} \occ{u[in+1, in+n]}{v} 
\\
&= (|u|/n-1)(m-1) + \sum_{w \in C^n} \alocc{u}{w} \occ{w}{v}.
\end{align*}
Using (\ref{alindisc}) we get,
$$
\occ{u}{v} < (|u|/n-1)(m-1) + \sum_{w \in C^n} \frac{|u|}{|v|}\left(\frac{1}{|C|^n} + \eps\right) \occ{w}{v}.
$$
Using that $\occ{w}{v} = \sum_{i=1}^{|w|-|v|} \uno(w[i,i+|v|] = v)$ we get,
$$
\occ{u}{v} < (|u|/n-1)(m-1) + \frac{|u|}{|v|}\left(\frac{1}{|C|^n} +
 \eps\right) \sum_{w \in C^n} \sum_{i=1}^{n-m} \uno(w[i,i+|v|] = v).
$$
Using Observation~\ref{cuentasubs} we get,
$$
\occ{u}{v} < (|u|/n-1)(m-1) + \frac{|u|}{|v|}\left(\frac{1}{|C|^n} + \eps\right) \sum_{i=1}^{n-m} |C|^{n-m}.
$$
Which is equivalent to 
$$
\occ{u}{v} < (|u|/n-1)(m-1) + \frac{|u|}{n}\left(\frac{1}{|C|^m} + |C|^{n-m}\eps\right)(n-m).
$$
And since $m<n$ we get,
$$
\occ{u}{v} < |u| \left(\frac{m-1}{n} + \frac{1}{|C|^m} + |C|^{n}\eps\right) - (m-1),
$$
as desired.
\end{proof}

The first paragraph in the proof above yields the following result.
\begin{observation}\label{concatapp}
Given a finite alphabet $C$, some $u, v, w \in C^*$ then $$\occ{uv}{w} \leq \occ{u}{w} + \occ{v}{w} + |w|-1.
$$
\end{observation}
Finally we recall the characterization of normality  that is seemingly easier 
than the actual definition, because instead of
asking for the limit  it asks for the limsup.

\begin{lemma}[Hot Spot Lemma, Piatetski-Shapiro 1951]%, Borwein and Bailey 2008]
\label{hotspot}
Let $x$ be an infinite word of symbols in alphabet $A$. 
Then, $x$ is normal if and only if there is positive constant~$C$ such that for all lengths~$\ell$
and  for every word $u$ of length $\ell$, 
$$
\limsup_{n\rightarrow\infty}   \frac{\occ{x[1,n]}{u}}{n}  < \frac{C}{|A|^\ell}.
$$
\end{lemma}

\section{Proof of  Theorem \ref{maintheo}}

%Recall that $A = \{0, 1, \ldots, b-1\}$, $\wh{A} = \{0, 1, \ldots, b\}$ and $r$ is the operator such that $r(v)$ is the word $v$ without all its $'b'$ symbols.
%\begin{theorem}\label{maintheo}
%Let $v \in A^{\omega}$ be a normal word then there exists some normal word $\wh{v} \in \wh{A}^{\omega}$ such that $r(\wh{v}) = v$.
%\end{theorem}
%\begin{proof}

 We construct inductively a sequence of nonempty finite 
substrings $\{v_i \}_{i\in \NN}$ of $v$ that verifies that 
$v_1v_2 \ldots v_k$ is a prefix of $v$ for all $k$ in $\NN$.
Suppose that we have already defined $v_1, v_2, \ldots, v_{n-1}$ and we want to define $v_n$. Let $L_{n-1} = |v_1v_2 \ldots v_{n-1}|$ be the total length of all substrings already defined. Since $v$ is normal, then $v \upharpoonleft L_{n-1}$ is also normal and consequently given 
$$
\eps_n = \frac{1}{(b+1)^{2^n}n}\frac{1}{3\max(b^nc_{2^n}, (b+1)^nc_{2^{n+1}})}
$$
 there exists a $k_n$ such that for all $k > k_n$ in $\NN$ we have 
$$
\Delta_{A,\ell_{2^{n+1}}}(v[L_{n-1}+1,L_{n-1}+k]) < \eps_n
$$
Take $t_n$ such that $t_n\ell_{2^n} > \max(k_n, \ell_{2^{n+1}})$ and define $v_n$ as 
$$v_n = v[L_{n-1}+1, L_{n-1}+t_n\ell_n]$$
It is clear that $v_1 v_2 \ldots v_n = v[1, L_n+t_n\ell_{2^n}]$ and thus is a prefix of $v$.\bigskip

Given $\{v_i \}_{i\in \NN}$ defined as above, we define the 
%extension 
expansion 
$\wh{v}$ as 
$$
\wh{v} = e_{2^1}(v_1)e_{2^2}(v_2) \ldots e_{2^i}(v_i) \ldots
$$
Since each $v_i$ has length $t_i\ell_{2^i}$ which is multiple of $\ell_{2^i}$, 
the 
%extension 
expansion is well defined. It follows easily that 
$$
r(\wh{v}) = r(e_{2^1}(v_1))r(e_{2^2}(v_2)) \ldots r(e_{2^i}(v_i)) \ldots = v.
$$ 
We claim that $\wh{v}$ is normal in base $b+1$.
We can write each $v_n$ as 
$$
v_n = v_{n, 1}v_{n, 2} \ldots v_{n, t_n}
$$
 where each $v_{n, i}$ satisfies $|v_{n, i}| = \ell_{2^n}$.
Fix $n \in \NN$ and $j \in \NN_0$ with $0 \leq j \leq t_{n+1}$, and define 
$$
v_{n+1}' = v_{n+1, 1}v_{n+1, 2} \ldots v_{n+1, j}
$$
 as the prefix of $v_{n+1}$ that consists of the first $j$ blocks of length $\ell_{2^{n+1}}$.
By definition of $v_n$, we have that 
\begin{equation}\label{deltavn}\Delta_{A,\ell_{2^{n+1}}}(v_n) < \eps_n
\end{equation}
 and since $v_nv'_{n+1}$ is a prefix of $v \upharpoonleft L_{n-1}$ of length greater than $k_n$ we have 
$$
\Delta_{A,\ell_{2^{n+1}}}(v_nv'_{n+1}) < \eps_n.
$$
Using Proposition~\ref{sufijo} we have that 
\begin{equation}\label{deltavpn}
\Delta_{A,\ell_{2^{n+1}}}(v'_{n+1}) < \frac{|v_nv'_{n+1}|+|v_n|}{|v'_{n+1}|}\eps_n.
\end{equation}
Now, by (\ref{deltavn}) and Proposition~\ref{mitad} we have that 
$$
\Delta_{A,\ell_{2^{n}}}(v_n) < b^{\ell_{2^n}}\eps_n
$$
 and applying Lemma~\ref{fund} we get 
\begin{equation}\label{deltaevn}
\Delta_{\wh{A},2^n}(e_{2^n}(v_n)) < b^{\ell_{2^n}}c_{2^n}\eps_n
\end{equation}
Similarly, applying Lemma~\ref{fund} to (\ref{deltavpn}) we get 
$$
\Delta_{\wh{A},2^{n+1}}(e_{2^{n+1}}(v'_{n+1})) 
< \frac{|v_nv'_{n+1}|+|v_n|}{|v'_{n+1}|}c_{2^{n+1}}\eps_n
$$
 and by Proposition~\ref{mitad} we conclude 
\begin{equation}\label{deltaevpn}
\Delta_{\wh{A},2^{n}}(e_{2^{n+1}}(v'_{n+1})) 
< \frac{|v_nv'_{n+1}|+|v_n|}{|v'_{n+1}|}(b+1)^{2^n}c_{2^{n+1}}\eps_n.
\end{equation}
Using Proposition~\ref{concat} with (\ref{deltaevn}) and (\ref{deltaevpn}) we get that 
\begin{equation}\label{cotadisc}
\Delta_{\wh{A},2^{n}}(e_{2^n}(v_n)e_{2^{n+1}}(v'_{n+1})) 
< 3\max(b^{\ell_{2^n}}c_{2^n}, (b+1)^{2^n}c_{2^{n+1}})\eps_n 
< \frac{1}{(b+1)^{2^n}n}.
\end{equation}
Notice that the bound does not depend on $j$. 
If $j = 0$ we get the special case 
\begin{equation}\label{cotadiscshort}
\Delta_{\wh{A},2^{n}}(e_{2^n}(v_n)) 
< \frac{1}{(b+1)^{2^n}n}.
\end{equation}
Now, we fix $u \in \wh{A}^m$ for some $m \in \NN$.
For $n \in \NN$ and $j \in \NN_0$ with $0 \leq j \leq t_n$, 
we define 
$$
L_{n,j} = |v_1v_2 \ldots v_{n}v_{n+1,1}v_{n+1,2}\ldots v_{n+1,j}|.
$$
Notice that $L_{n, t_n} = L_{n+1, 0}$. We define $L_{0,0} = 0$.
Given some $M \in \NN$ with $M > L_{1,0}$,
 there exists some $n, j \in \NN$ with $n > 1$ such that 
\begin{equation}\label{condm}
L_{n, j-1} \leq M \leq L_{n,j}.
\end{equation}
By Observation~\ref{concatapp} we get
\begin{equation}\label{cotaparcialraw}
\begin{gathered} \occ{\wh{v}[1,M]}{u} \leq \occ{\wh{v}[1,L_{n,j}]}{u} \leq 
\\ \left(\sum_{i=1}^{n-1} \occ{e_{2^i}(v_i)}{u}\right) + 
\occ{e_{2^n}(v_n)e_{2^{n+1}}(v_{n+1,1}) \ldots e_{2^{n+1}}(v_{n+1,j})}{u} + (n-1)(|u|-1).
\end{gathered}
\end{equation}
Given that we have (\ref{cotadiscshort}) for each term of the sum, 
we can apply Proposition~\ref{alinanoalin} and we get the bound
$$
\sum_{i=1}^{n-1} \occ{e_{2^i}(v_i)}{u} 
\leq \sum_{i=1}^{n-1} |e_{2^i}(v_i)| \left(\frac{|u|-1}{2^i} + \frac{1}{(b+1)^{|u|}} 
+ \frac{(b+1)^{2^i}}{(b+1)^{2^i}i} \right) - (n-1)(|u|-1).
$$
Noticing that 
$$\frac{|u|-1}{2^i}+\frac{1}{i} \rightarrow 0
\text{ as $i \rightarrow \infty$ }
$$
there exists some $i_0$ such that for all $i > i_0$ we have 
\begin{equation}
\label{cotasumaparcial}\frac{|u|-1}{2^i}+\frac{1}{i} \leq \frac{1}{(b+1)^{|u|}}.
\end{equation}
If $M$ is sufficiently large, we will have $n > i_0$ and then we can split the sum and get
\begin{equation}
\begin{split}
\sum_{i=1}^{n-1} \occ{e_{2^i}(v_i)}{u} 
\leq &\sum_{i=1}^{i_0} |e_{2^i}(v_i)| \left(\frac{|u|-1}{2^i} + \frac{1}{(b+1)^{|u|}} + \frac{1}{i} \right) + \\ 
&\sum_{i=i_0+1}^{n-1} |e_{2^i}(v_i)| \left(\frac{|u|-1}{2^i} + \frac{1}{(b+1)^{|u|}} + \frac{1}{i} \right) - \\
 &(n-1)(|u|-1).\end{split}
 \end{equation}
Calling 
$$
\delta = \sum_{i=1}^{i_0} |e_{2^i}(v_i)| 
\left(\frac{|u|-1}{2^i} + \frac{1}{(b+1)^{|u|}} + \frac{1}{i} \right)
$$
 (notice that $\delta$ does not depend on $M$) and using (\ref{cotasumaparcial}) we get
$$
\sum_{i=1}^{n-1} \occ{e_{2^i}(v_i)}{u} \leq \delta + \frac{2}{(b+1)^{|u|}}
\left(\sum_{i=i_0+1}^{n-1} |e_{2^i}(v_i)| \right)- (n-1)(|u|-1).
$$
Using that $|e_{2^i}(v_i)| = L_{i,0}-L_{i-1,0}$ we can reduce this to
\begin{equation}\label{cotaprimero}
\sum_{i=1}^{n-1} \occ{e_{2^i}(v_i)}{u} \leq \delta + 
(L_{n-1,0}-L_{i_0,0})\frac{2}{(b+1)^{|u|}}- (n-1)(|u|-1).
\end{equation}
Having (\ref{cotadisc}) and using Proposition~\ref{alinanoalin} 
with the second term of (\ref{cotaparcialraw}) we get
\begin{align*}
&\occ{e_{2^n}(v_n)e_{2^{n+1}}(v_{n+1,1}) \ldots e_{2^{n+1}}(v_{n+1,j})}{u} \leq
\\
&(L_{n,j}-L_{n-1,0}) \left(\frac{|u|-1}{2^n} + \frac{1}{(b+1)^{|u|}} + \frac{(b+1)^{2^n}}{(b+1)^{2^n}n} \right) - (|u|-1).
\end{align*}
Since $n > i_0$ we get
\begin{equation}\label{cotasegundo}\occ{
e_{2^n}(v_n)e_{2^{n+1}}(v_{n+1,1}) \ldots e_{2^{n+1}}(v_{n+1,j})}{u} 
\leq (L_{n,j}-L_{n-1,0}) \left(\frac{2}{(b+1)^{|u|}}\right).
\end{equation}
Using (\ref{cotaprimero}) and (\ref{cotasegundo}) in (\ref{cotaparcialraw}) we get
$$
\occ{\wh{v}[1,M]}{u} \leq \delta + (L_{n,j} - L_{i_0,0})\frac{2}{(b+1)^{|u|}}.
$$
Dividing both sides by $\occ{\wh{v}[1,M]}{u} = M$ we get
\begin{equation}\label{cotahotraw}
\frac{\occ{\wh{v}[1,M]}{u}}{M} \leq \frac{\delta}{M} + \frac{L_{n,j} - L_{i_0,0}}{M}\frac{2}{(b+1)^{|u|}}.
\end{equation}
By (\ref{condm}) we have that 
$$L_{n,j} - M \leq L_{n,j} - L_{n,j-1} = \wh{\ell}_{2^{n+1}}.
$$
By construction of $v_n$,
$$\ell_{2^{n+1}} \leq |v_n|.
$$ 
Then, since $e_{2^n}(v_n)$ is a substring of $\wh{v}[1,M]$ we get that 
$$
\wh{\ell}_{2^{n+1}} \leq |e_{2^n}(v_n)| \leq M.
$$
Which gives us the bound $L_{n,j} \leq 2M$. Using this in (\ref{cotahotraw}) we get
\begin{equation}\frac{\occ{\wh{v}[1,M]}{u}}{M} \leq \frac{\delta}{M} 
+ \frac{2M - L_{i_0,0}}{M}\frac{2}{(b+1)^{|u|}} < \frac{\delta}{M} + \frac{4}{(b+1)^{|u|}}.
\end{equation}
Taking limit superior as $M \rightarrow \infty$ and since $\delta$ does not depend on $M$ we get
\begin{equation}
\limsup_{M\rightarrow \infty} \frac{\occ{\wh{v}[1,M]}{u}}{M} \leq \frac{4}{(b+1)^{|u|}}.
\end{equation}
Since this bound is valid for all $u \in \wh{A}^*$, using the Lemma~\ref{hotspot} (Hot Spot Lemma)
with $C=4$ follows that $\wh{v}$ is normal. 
Therefore, we constructed a normal word $\wh{v}$ such that $r(\wh{v}) = v$ as desired.
This completes the proof of Theorem~\ref{maintheo}.\hfill \qedsymbol

%\end{proof}

\section{Some remarks about the proof of Theorem~\ref{maintheo}}

\subsection{On the choice of $w_n$}
We can study how flexible is the construction of the proof 
on the choice of the sequence $(w_n)$. We wonder what 
%others 
other
sequences we can choose so that the proof remains valid. Looking at the proof, 
the only places where we use the explicit construction of $(w_n)$ is 
%on 
in
Lemma~\ref{sumafund} and Lemma~\ref{fund}. The property of the sequence we are using is that 
$$\Delta_{n,\wh{A}} w_n = 0$$
This means that we can change the $w_n$ for some other 
sequence satisfying this property. In fact, if we only have that the discrepancy of $w_n$ is small, namely 
$$\Delta_{n,\wh{A}} w_n < \delta_n$$
we can, with a little more of work, obtain 
a bound similar to that of to Lemma~\ref{fund} 
but also involving the $\delta_n$. 
Then, we can use this bound 
in 
the proof of the Theorem~\ref{maintheo}. 
If we choose the $\delta_n$ to be  small enough 
(and maybe depending of the $\eps_n$) we can adapt the proof to work for this new sequence $(w_n)$.
Any normal number $z \in \wh{A}^{\omega}$ can be
split it  in consecutive strings  $z_1, z_2, \ldots, z_n, \ldots$ 
such that 
$$
z = z_1z_2 \ldots z_n \ldots
$$
and for each $n$, the word $z_n$ satisfies
$$
\Delta_{n, \wh{A}}(z_n) < \delta_n.
$$
If the lengths of $z_n$ do not grow larger than exponential on $n$, 
we can use this sequence $(z_n)$ 
as an alternative for $(w_n)$ to expand normal numbers. 
This means that we can do the process to expand 
a normal number to $\wh{A}$ with substrings of any normal number in $\wh{A}$ 
that has a partition into substrings with this property.

\subsection{On the computability of the construction}
If we know the convergence rates of the normal word to expand, we can calculate $e_n$ for all $n \in \NN$ 
and we can easily compute the expanded word. If we don't know anything 
about the convergence rates, we can still compute 
the expanded word with a finite-injury priority method \cite{rogers1987theory}, 
but we will not know how good will be our approximation at each step of the algorithm.

\bibliographystyle{plain}
\bibliography{ProblemaDeExtensionBib}
\end{document}